\documentclass[12pt, a4paper]{article}
\usepackage{lipsum}
\usepackage{authblk}
\usepackage{amsfonts}
\usepackage{amssymb}
\usepackage{amscd}
\usepackage{mathrsfs}
\usepackage{amsmath,amssymb, amsthm}
\usepackage{graphicx}
\usepackage{color}
\usepackage{authblk}
\usepackage{amssymb}
\usepackage{indentfirst,latexsym}
\usepackage[bookmarks=false]{hyperref}
\usepackage{amsmath}

\newtheorem{theorem}{Theorem}[section]
\newtheorem{pro}[theorem]{Proposition}

\newtheorem{lem}[theorem]{Lemma}

\newtheorem{coro}[theorem]{Corollary}

\theoremstyle{definition}
\newtheorem{defi}[theorem]{Definition}
\newtheorem{exam}[theorem]{Example}
\newtheorem{problem}[theorem]{Problem}

\newtheorem{rem}[theorem]{Remark}

\def\Cos{\hbox{\rm Cos}}
\def\Z{\mathbb{Z}}

\def\Aut{\hbox{\rm Aut}}

\def\Ga{\Gamma}

\usepackage{enumerate}
\long\def\delete#1{}

\usepackage{xcolor}
\usepackage[normalem]{ulem}

\title{Perfect codes in vertex-transitive graphs}

\author[a]{Yuting Wang}
\author[b]{Junyang Zhang\footnote{Corresponding author}}
\affil[a]{{\small School of Mathematical Sciences, Capital Normal University, Beijing 100048, P. R. China}}
\affil[b]{{\small School of Mathematical Sciences, Chongqing Normal University, Chongqing 401331, P. R. China}}
\date{}

\begin{document}

\openup 0.5\jot
\maketitle
\footnotetext{E-mail: yiniyiting@163.com (Yuting Wang), jyzhang@cqnu.edu.cn (Junyang Zhang)}

\vspace{-10mm}
\begin{abstract}
Given a graph $\Gamma$, a perfect code in $\Gamma$ is an independent set $C$ of vertices of $\Gamma$ such that every vertex outside of $C$ is adjacent to a unique vertex in $C$, and a total perfect code in $\Gamma$ is a set $C$ of vertices of  $\Gamma$ such that every vertex of $\Gamma$ is adjacent to a unique vertex in $C$. To study (total) perfect codes in vertex-transitive graphs, we generalize the concept of subgroup (total) perfect code of a finite group introduced in \cite{HXZ18} as follows: Given a finite group $G$ and a subgroup $H$ of $G$, a subgroup $A$ of $G$ containing $H$ is called a subgroup (total) perfect code of the pair $(G,H)$ if there exists a coset graph $\Cos(G,H,U)$ such that the set consisting of left cosets of $H$ in $A$ is a (total) perfect code in $\Cos(G,H,U)$. We give a necessary and sufficient condition for a subgroup $A$ of $G$ containing $H$ to be a (total) perfect code of the pair $(G,H)$ and generalize a few known results of subgroup (total) perfect codes of groups. We also construct some examples of subgroup perfect codes of the pair $(G,H)$ and propose a few problems for further research.


\end{abstract}

\section{Introduction}
In this paper, all groups considered are finite, and all graphs considered are finite, undirected and simple.
Let $\Gamma$ be a graph with vertex set $V(\Gamma)$ and edge set $E(\Gamma)$, and let $t$ be a positive integer. A subset $C$ of $V(\Gamma)$ is called \cite{Big, Kr86} a  \emph{perfect $t$-code} in $\Gamma$ if every vertex of $\Gamma$ is at distance no more than $t$ to exactly one vertex in $C$, where the \emph{distance} in $\Ga$ between two vertices is the length of a shortest path between the two vertices or $\infty$ if there is no path in $\Gamma$ joining them. A perfect $1$-code is usually called a {\em perfect code}. Equivalently, a subset $C$ of $V(\Gamma)$ is a perfect code in $\Gamma$ if $C$ is an independent set of $\Gamma$ and every vertex in $V(\Gamma) \setminus C$ has exactly one neighbor in $C$. A subset $C$ of $V(\Gamma)$ is said to be a \emph{total perfect code} \cite{Zhou2016} in $\Gamma$ if every vertex of $\Gamma$ has exactly one neighbor in $C$. It is obvious that a total perfect code in $\Gamma$ induces a matching in $\Gamma$ and therefore has even cardinality.
In graph theory, a perfect code in a graph is also called an \emph{efficient dominating set} \cite{DeS} or \emph{independent perfect dominating set }\cite{Le}, and a total perfect code is called an \emph{efficient open dominating set} \cite{HHS}.

The concept of $t$-perfect codes in graphs were firstly introduced by Biggs \cite{Big} as a generalization of the classical concept \emph{perfect $t$-error-correcting code} in coding theory \cite{Heden1,HK18,Va75,MS77}. For a set $A$ (usually with an algebraic structure such as group, ring, or field), we use $A^{n}$ to denote the $n$-fold Cartesian product of $A$. In coding theory, $A$ is called an \emph{alphabet} and elements in $A^{n}$ are called \emph{words} of length $n$ over $A$. A \emph{code} $C$ over an alphabet $A$ is simply a subset of $A^{n}$, and every word in $C$ is called a \emph{codeword}. The Hamming distance of two words in $A^{n}$ is the number of positions in which they differ. A code $C$ over $A$ is called a \emph{perfect $t$-error-correcting Hamming code} if every word in $A^{n}$ is at Hamming distance no more than $t$ to exactly one codeword of $C$.
The \emph{perfect $t$-error-correcting Lee code} over $A$ is defined in a similar way if $A$ is the ring $\Z_m$ of integers $\pmod m$, where the Lee distance of two words $x=(x_{1},x_{2},\cdots,x_{n}),y=(y_{1},y_{2},\cdots,y_{n})\in \Z_{m}^{n}$ is defined as follows:
  $d_{L}(x,y)=\sum\limits_{i=1}^{n}\min(|x_{i}-y_{i}|,m-|x_{i}-y_{i}|)$.
Recall that the \emph{Hamming graph} $H(n,m)$ is the Cartesian product of $n$ copies of the complete graph $K_m$ and the \emph{grid-like graph}
$L(n, m)$ is the Cartesian product of $n$ copies of the $m$-cycle $C_{m}$. It is obvious that the perfect $t$-error-correcting Hamming codes over an alphabet of cardinality $m$ are precisely the perfect $t$-codes in $H(n,m)$. Similarly, the perfect $t$-error-correcting Lee codes over $\Z_m$ ($m\geq3$) are precisely the perfect $t$-codes in $L(n, m)$.

A graph $\Gamma$ is called $G$-\emph{vertex-transitive} if $G$ is a subgroup of $\Aut(\Gamma)$ acting transitively on $V(\Gamma)$. In particular, a $G$-\emph{vertex-transitive} graph is called a \emph{Cayley graph} on $G$ if $G$ acts \emph{freely} on the vertex set (nonidentity elements fix no vertex). It is well known and easy to check that both $K_m$ and $C_m$ are Cayley graphs on the cyclic group $\Z_m$. Therefore $H(n,q)$ and $L(n,m)$ are both Cayley graphs on the group $\Z_m^n$. Thus perfect $t$-codes in Cayley graphs are generalization of perfect $t$-error-correcting Hamming codes or Lee codes.

Perfect codes in Cayley graphs have received considerable attention in recent
years; see \cite[Section 1]{HXZ18} for a brief survey and \cite{DSLW16, FHZ, Ta13, Z15,ZZ2021} for a few recent papers. In particular, perfect codes in Cayley graphs which are subgroups of the underlying groups are especially interesting since they are generalizations of perfect linear codes \cite{Va75} in the classical setting. Another interesting avenue of research is to study when a given subset of a group is a perfect code in some Cayley graph of the group. In this regard the following concepts were introduced by Huang et al. in \cite{HXZ18}: A subset $C$ of a group $G$ is called a {\em (total) perfect code} of $G$ if there exists a Cayley graph of $G$ which admits $C$ as a (total) perfect code; a (total) perfect code of $G$ which is also a subgroup of $G$ is called a {\em subgroup (total) perfect code} of $G$. Huang
et al. \cite{HXZ18} established a sufficient and necessary
condition for the normal subgroups of a given group to be subgroup (total) perfect codes, and proved that every normal subgroup of a group of odd order or odd index is a subgroup perfect code. Ma et al. \cite{MWWZ19} proved that all subgroups of a group are subgroup perfect codes if and only if this group does not contain elements of order $4$.
Very recently, Zhang and Zhou \cite{Z15} generalized several results about normal subgroups in \cite{HXZ18} to general subgroups, and in particular they proved that every subgroup of a group of odd order or odd index is a subgroup perfect code.

Although every Cayley graph is vertex-transitive, there exist other vertex-transitive graphs that are not Cayley graphs. Somewhat surprisingly, there are very few known results on the perfect codes in vertex-transitive graphs in the literature. This motivates us to write the present paper. It is well known that a graph is $G$-vertex transitive if and only if it can be represented as a coset graph $\Cos(G,H,U)$ (see Section \ref{sec:pre} for the details). To study the perfect codes in vertex-transitive graphs, we generalize the concept subgroup (total) perfect code of a finite group as follows:
Given a finite group $G$ and a subgroup $H$ of $G$, a subgroup $A$ of $G$ containing $H$ is called a \emph{subgroup (total) perfect code of the pair $(G,H)$} if there exists a coset graph $\Cos(G,H,U)$ such that the set consisting of left cosets of $H$ in $A$ is a (total) perfect code in $\Cos(G,H,U)$. In this paper, we give a necessary and sufficient condition for a subgroup $A$ of $G$ containing $H$ to be a (total) perfect code of the pair $(G,H)$ and generalize a few known results of subgroup (total) perfect codes of groups.

The rest of the paper is organized as follows. In Section \ref{sec:pre}, we recall the definition of coset graph and give a characterization of the relationship between subgroup perfect codes  and subgroup total perfect codes of a pair $(G,H)$. In Section \ref{sec:pc}, we prove that $A$ is a perfect code of a pair $(G,H)$ if and only if there exists a left transversal $X$ of $A$ in $G$ such that $XH=HX^{-1}$ (Theorem \ref{lt}). Based on Theorem \ref{lt}, we generalize a few results about subgroup perfect codes in \cite{ZZ2021}. In Section \ref{sec:tpc}, we deduce a few results on total perfect codes which are parallel to some results about perfect codes in Section \ref{sec:pc}. In Section \ref{sec:ep}, we construct several examples and propose a few problems for further research. In particular, we show that $S_{n-1}$ is a perfect code of $(S_{n},S_{3})$ for every positive integer $n\geq5$.
\section{Preliminaries}
\label{sec:pre}

For a group $G$, we write $H\leq G$ to signify that $H$ is a subgroup of $G$ and we set $A/_{\ell}H:=\{aH\mid a\in A\}$ for all subset $A$ of $G$. For more group-theoretic terminology and notation used in the paper, please refer to  \cite{KS2004}. The following proposition gives a nice way to represent vertex-transitive graphs. For the proof of this proposition, see \cite{Lor}.

\begin{pro}\label{cos}
Let $G$ be a group and $H\leq G$. Let $U$ be a union of some double cosets of $H$ in $G$ such that $H\cap U=\emptyset$ and $U^{-1}=U$. Define a graph $\Gamma=\Cos(G,H,U)$ as follows: the vertex set of $\Gamma$ is $G/_{\ell}H$, and two vertices $g_{1}H$ and $g_{2}H$ are adjacent if and only if $g_{1}^{-1}g_{2}\in U$. Then we have
\begin{enumerate}
  \item $\Gamma$ is a well defined graph and its valency is the number of left cosets of $H$ in $U$;
  \item $G$ acts transitively on the vertex-set of $\Gamma$ by left multiplication and the kernel of this action is the core of $H$ in $G$;
  \item every vertex-transitive graph can be represented as $\Cos(G,H,U)$ for some $G$, $H$ and $U$.
\end{enumerate}
\end{pro}
The graph $\Gamma=\Cos(G,H,U)$ defined in Proposition \ref{cos} is usually called a \emph{coset graph} on $G/_{\ell}H$. It is straightforward to check that the neighbourhood of $H$ in $\Gamma$ is $U/_{\ell}H$ and $\Gamma$ is connected if and only if $G=\langle U\rangle$.
\begin{defi}
\label{defi}
Let $G$ be a group, $H\leq G$ and $A\subseteq G$. If there exists a coset graph $\Cos(G,H,U)$ on $G/_{\ell}H$ admitting a (total) perfect code $A/_{\ell}H$, then $A$ is called a (total) perfect code of the pair $(G,H)$. If further $H\leq A\leq G$, then $A$ is called a subgroup (total) perfect code of the pair $(G,H)$.
\end{defi}
\begin{rem}
It is obvious that $\Cos(G,1,U)$ is a Cayley graph on $G$. Thus $A$ is a subgroup (total) perfect code of $(G,1)$ if and only if it is a subgroup (total) perfect code of $G$. Therefore the concept of subgroup (total) perfect code of $(G,H)$ is a generalization of the concept of subgroup (total) perfect code of $G$.
\end{rem}
The following lemma gives a characterization of the relationship between subgroup perfect codes  and subgroup total perfect codes of a pair $(G,H)$.
\begin{lem}
\label{totallem}
Let $G$ be a group and $H\leq A\leq G$. Then $A$ is a total perfect code of $(G,H)$ if and only if $A$ is a perfect code of $(G,H)$ and there exists an element $x\in N_{A}(H)\setminus H$ such that $x^{2}\in H$.
\end{lem}
\begin{proof}
$\Rightarrow$) Suppose that $A$ is a total perfect code of $(G,H)$. Then there exists a coset graph $\Cos(G,H,U)$ on $G/_{\ell}H$ such that $A/_{\ell}H$ is a total perfect code of $\Cos(G,H,U)$. In particular, $A/_{\ell}H$ induces a matching in $\Cos(G,H,U)$. Thus there exists $x\in A\setminus H$ such that $xH$ is the unique vertex in $A/_{\ell}H$ which is adjacent to $H$. By the definition of coset graph, $U$ is a union of some double cosets of $H$, $U^{-1}=U$
and $x\in U$. Therefore $hx,hx^{-1}\in U$ for every $h\in H$. It follows that $hxH$ and $hx^{-1}H$ are both neighbours of $H$ in $\Cos(G,H,U)$. Note that $hxH,hx^{-1}H\in A/_{\ell}H$. By the uniqueness of $xH$, we have $hxH=hx^{-1}H=xH$. Therefore $(HxH)^{-1}=HxH=xH$. Set $W=U\setminus HxH$. Then $W^{-1}=W$ and $A/_{\ell}H$ is a perfect code of the coset graph $\Cos(G,H,W)$. Therefore $A$ is a perfect code of $(G,H)$. Since $x^{-1}H=xH$, we have $x^{2}\in H$. Recall that $x\in A\setminus H$. Since $hxH=xH$ for every $h\in H$, we get $x^{-1}hx\in H$ and it follows that $x\in N_{A}(H)\setminus H$.

$\Leftarrow$) Suppose that $A$ is a perfect code of $(G,H)$  and there exists an element $x\in N_{A}(H)\setminus H$ such that $x^{2}\in H$. Then there exists a coset graph $\Cos(G,H,W)$ on $G/_{\ell}H$ such that $A/_{\ell}H$ is a perfect code of $\Cos(G,H,W)$. Since $x\in N_{A}(H)\setminus H$ and $x^{2}\in H$, we have $(HxH)^{-1}=HxH=xH=x^{-1}H$. Set $U=W\cup HxH$. Then $U^{-1}=U$ and $A/_{\ell}H$ is a total perfect code of $\Cos(G,H,U)$. Therefore $A$ is a total perfect code of $(G,H)$.
\end{proof}
\section{Subgroup perfect codes of $(G,H)$}
\label{sec:pc}

In this section, we deduce some general results about subgroup perfect codes of a pair $(G,H)$. Our first result below gives a necessary and sufficient condition for a subgroup $A$ of $G$ containing $H$ to be a perfect code of $(G,H)$.
\begin{theorem}
\label{lt}
Let $G$ be a group and $H\leq A\leq G$. Then $A$ is a perfect code of $(G,H)$ if and only if there exists a left transversal $X$ of $A$ in $G$ such that $XH=HX^{-1}$.
\end{theorem}
\begin{proof}
$\Rightarrow$) Let $A$ be a perfect code of $(G,H)$. Then there exists a coset graph $\Gamma:=\Cos(G,H,U)$ such that $A/_{\ell}H$ is a perfect code of $\Gamma$. By the definition of coset graph, we get $H\cap U=\emptyset$ and $U^{-1}=U$. Assume $|G:A|=n$ and let $T=\{1,t_{1},\ldots,t_{n-1}\}$ be a left transversal of $A$ in $G$. Then $t_{i}\notin A$ for any $i\in \{1,\ldots,n-1\}$. Since $H\leq A$, we get $t_{i}^{-1}H\notin A/_{\ell} H$. Since $A/_{\ell}H$ is a perfect code of $\Gamma$, $t_{i}^{-1}H$ is adjacent to a unique vertex in $A/_{\ell}H$.
Therefore there is a unique $a_{i}H\in A/_{\ell}H$ such that
$t_{i}a_{i}\in U$. Set $X=\{1,t_{1}a_{1},\ldots,t_{n-1}a_{n-1}\}$. Then $X$ is a left transversal of $A$ in $G$ and $X\setminus\{1\}\subseteq U$. Since $U$ is a union of some double cosets of $H$ in $G$, we have $H(X\setminus\{1\})H\subseteq U$. We will further prove $U\subseteq(X\setminus\{1\})H$. Take an arbitrary $u\in U$. Since $T$ is a left transversal of $A$ in $G$, $u$ can be uniquely written as $u=ta$ where $t\in T$ and $a\in A$. It follows that $ta\in U$ and therefore $t^{-1}H$ and $aH$ are adjacent in $\Gamma$. Since $A/_{\ell}H$ is a perfect code of $\Gamma$, $A/_{\ell}H$ is an independent set of $\Gamma$. Therefore $t\notin A$. Since $T=\{1,t_{1},\ldots,t_{n-1}\}$, we have $t=t_{i}$ for some $i\in\{1,\ldots,n-1\}$. By the uniqueness of $a_{i}H$, we have $aH=a_{i}H$ and then $u=t_{i}a_{i}h$ for some $h\in H$. Therefore $u\in (X\setminus\{1\})H$ and it follows that $U\subseteq (X\setminus\{1\})H$. Now we have proved that $H(X\setminus\{1\})H\subseteq U\subseteq(X\setminus\{1\})H$. Therefore $U=H(X\setminus\{1\})H=(X\setminus\{1\})H$.
Since $U^{-1}=U$, we have $H(X\setminus\{1\})^{-1}=U^{-1}=U=(X\setminus\{1\})H$ and it follows that $XH=HX^{-1}$.

$\Leftarrow$)
Let $X$ be a left transversal of $A$ in $G$ such that $XH=HX^{-1}$. Then $HXH=HHX^{-1}=HX^{-1}$, $(HXH)^{-1}=HX^{-1}H=XH$ and it follows that
\begin{equation*}
(HXH)^{-1}=HXH=XH=HX^{-1}.
\end{equation*}
Since $X$ is a left transversal of $A$ in $G$, $X\cap A$ contains a unique element, say $y$. Since $XH=HX^{-1}$, for every $h\in H$ there exists $w\in X$ such that $hy^{-1}\in wH$.
Since $H\leq A$ and $hy^{-1}\in A$, we have $w\in A$. Therefore $w=y$ and it follows that $Hy^{-1}=Hy$. Set $U:=H(X\setminus\{y\})H$. Since $(HXH)^{-1}=HXH=XH=HX^{-1}$ and $Hy^{-1}=yH$, we have $U^{-1}=U=(X\setminus\{y\})H=H(X\setminus\{y\})^{-1}$. Furthermore, $A\cap U=H\cap U=\emptyset$. In particular, we obtain a coset graph $\Gamma:=\Cos(G,H,U)$ on $G/_{\ell}H$. Since $A\cap U=\emptyset$,
$a^{-1}b\notin U$ for any $a,b\in A$ and it follows that $A/_{\ell}H$ is an independent set of $\Gamma$. Now consider an arbitrary vertex $gH\in G/_{\ell}H$ with $g\notin A$. Since $X$ is a left transversal of $A$ in $G$, $g^{-1}$ can be uniquely written as $g^{-1}=xa^{-1}$ where $x\in X$ and $a\in A$. Since $g^{-1}\notin A$, we obtain $x\neq y$. Therefore $x\in U$, that is, $g^{-1}a\in U$. It follows that $aH$ is adjacent to $gH$ in $A/_{\ell}H$. If there is $b\in A$ such that $bH$ is also adjacent to $gH$ in $A/_{\ell}H$, then $g^{-1}b\in U$. Since $U=(X\setminus\{y\})H$, we have $g^{-1}b=zh$ for some $z\in X$ and $h\in H$. Therefore $g^{-1}=zhb^{-1}$. Note that $hb^{-1}\in A$. By the uniqueness of the factorization $g^{-1}=xa^{-1}$, we have $z=x$ and $a^{-1}=hb^{-1}$. Thus $aH=bH$ and it follows that $aH$ is the unique vertex in $A/_{\ell}H$ which is adjacent to $gH$. Therefore $A/_{\ell}H$ is a perfect code of $\Gamma$, that is, $A$ is a perfect code of $(G,H)$.
\end{proof}
If we replace the word `left' with `right' in Theorem \ref{lt}, this theorem still holds. Actually, we have the following theorem.
\begin{theorem}
\label{rt}
Let $G$ be a group and $H \leq A \leq G$. Then $A$ is a perfect code of $(G,H)$ if and only if there exists a right transversal $Y$ of $A$ in $G$ such that $Y^{-1}H=HY$.
\end{theorem}
\begin{proof}
By Theorem \ref{lt}, $A$ is a perfect code of $(G,H)$ if and only if there exists a left transversal $X$ of $A$ in $G$ such that $XH=HX^{-1}$. Replacing $X^{-1}$ by $Y$, we have that $A$ is a perfect code of $(G,H)$ if and only if there exists a right transversal $Y$ of $A$ in $G$ such that $Y^{-1}H=HY$.
\end{proof}
Theorem \ref{lt} has several interesting corollaries. The first one below is obvious and we omit its proof.
\begin{coro}
\label{lrt}
Let $G$ be a group and $H \leq A \leq G$. Let $X$ be a left transversal of $A$ in $G$. If $XH=HX^{-1}$, then  $A/_{\ell}H$ is a perfect code of the coset graph $\Cos(G,H,U)$ where $U=H(X\setminus A)H$.
\end{coro}
\begin{coro}
\label{conj}
Let $G$ be a group and $H\leq A\leq G$. If $A$ is a perfect code of $(G,H)$, then for any $g \in G$, $g^{-1}Ag$ is a perfect code of $(G,g^{-1}Hg)$.
\end{coro}
\begin{proof}
By the necessity of Theorem \ref{lt}, there exists a left transversal $X$ of $A$ in $G$ such that $XH=HX^{-1}$. Set $Y=g^{-1}Xg$. Then $Y$ is a left transversal of $g^{-1}Ag$ in $G$ and $Yg^{-1}Hg=g^{-1}HgY^{-1}$. By the sufficiency of Theorem \ref{lt}, $g^{-1}Ag$ is a perfect code of $(G,g^{-1}Hg)$.
\end{proof}
\begin{coro}
\label{sub}
Let $G$ be a group and $H \leq A \leq L\leq G$. If $A$ is a perfect code of $(G,H)$, then $A$ is a perfect code of $(L,H)$.
\end{coro}
\begin{proof}
Let $A$ be a perfect code of $(G,H)$. By Theorem \ref{lt}, $A$ has a left transversal $X$ in $G$ such that $XH=HX^{-1}$. Therefore $G=XA$ and $|X\cap A|=1$. Set $Y=X \cap L$. Since $A \leq L\leq G$, we have $L=G\cap L= XA\cap L=(X\cap L)A=YA$ and $|Y\cap A|=|X\cap A|=1$. Therefore $Y$ is a left transversal of $A$ in $L$. Since $H \leq L$, we get $YH=(X\cap L)H=XH\cap LH=XH\cap L$. Since $XH=HX^{-1}$, it follows that
$YH=XH\cap L=HX^{-1}\cap HL^{-1}=H(X^{-1}\cap L^{-1})=HY^{-1}$.  Therefore, by Theorem \ref{lt}, $A$ is a perfect code of $(L,H)$.
\end{proof}
It is natural to consider the opposite of Corollary \ref{sub}. The following theorem is a preliminary exploration of that.
\begin{theorem}
\label{KL}
Let $G$ be a group admitting a normal subgroup $K$ and a subgroup $L$ such that $G=KL$ and $K\cap L=\{1\}$. Let $H\leq A\leq L$. Then $A$ is a perfect code of $(G,H)$ if and only if $A$ is a perfect code of $(L,H)$.
\end{theorem}
\begin{proof}
The necessity follows Corollary \ref{sub}. Now we prove the sufficiency. Suppose that $A$ is a perfect code of $(L,H)$. By Theorem \ref{lt}, there exists a left transversal $Y$ of $A$ in $L$ such that $YH=HY^{-1}$. Since $Y$ is a left transversal of $A$ in $L$, we have $L=YA$ and $|L|=|Y||A|$. Set $X=KY$. Since $G=KL$ and $K\cap L=\{1\}$, we have $G=KYA=XA$ and
\begin{equation*}
|G|=|KL|=|K||L|=|K||YA|=|K||Y||A|=|KY||A|=|X||A|.
\end{equation*}
Therefore $X$ is a left transversal of $A$ in $G$. Since $K$ is a normal subgroup of $G$ and $YH=HY^{-1}$, we have $XH=KYH=YHK=HY^{-1}K=HX^{-1}$. By Theorem \ref{lt}, $A$ is a perfect code of $(G,H)$.
\end{proof}
We use $S_{1}\dot{\cup}S_{2}$ to denote the union of two disjoint sets $S_{1}$ and $S_{2}$, and $\dot{\cup}_{i=1}^{m}S_{i}$ the union of pairwise disjoint sets $S_{1},\ldots,S_{m}$.
The following theorem generalizes the necessary part of \cite[Theorem 3.1]{ZZ2021}.
\begin{theorem}
\label{necessity}
Let $G$ be a group and $H\leq A \leq G$. If $A$ is a perfect code of $(G,H)$, then for any $g\in G$ either the left coset $gA$ contains an element $x$ such that $x^2\in b^{-1}Hb$ for some $b\in A$ or $A\{g,g^{-1}\}A=\dot{\cup}_{i=1}^{m}g_{i}A$ for some $g_{1},\ldots,g_{m}\in G$ where $m$ is an even integer.
\end{theorem}
\begin{proof}
Suppose that $A$ is a perfect code of $(G,H)$. By Theorem \ref{lt}, $A$ has a left transversal $T$ in $G$ such that $TH=HT^{-1}$.

Take an arbitrary $g\in G$. If $g\in A$, then $x\in gA$ and $x^{2}\in H$ for each $x\in H$. Now assume $g\in G\setminus A$. It is obvious that $A\{g,g^{-1}\}A$ is a disjoint union of some left cosets of $A$ in $G$. Suppose that $A\{g,g^{-1}\}A=\dot{\cup}_{i=1}^{m}g_{i}A$ for some $g_{1},\ldots,g_{m}\in G$ where $m$ is an odd integer. It suffices to prove that $gA$ contains an element $x$ such that $x^2\in b^{-1}Hb$ for some $b\in A$. Since $T$ is a left transversal of $A$ in $G$, there is a unique $x_i\in T$ such that $g_{i}A=x_{i}A$ for each $1\leq i\leq m$. Set $X=\{x_1, x_2, ...,x_m\}$. Then $A\{g,g^{-1}\}A=\dot{\cup}_{i=1}^{m}x_{i}A=XA$.
Since $X\subseteq T$ and $H\leq A$, we get $XH=TH\cap XA$. Therefore  $HX^{-1}=HT^{-1}\cap AX^{-1}$.
Since $XA=A\{g,g^{-1}\}A=(A\{g,g^{-1}\}A)^{-1}=AX^{-1}$ and $TH=HT^{-1}$, we have $XH=HX^{-1}$ and it follows that $(HXH)^{-1}=HXH=XH$.

Let $Y$ be a subset of $X$ of minimal cardinality such that $HXH=HYH$. Then $HyH\cap Hy'H=\emptyset$ for any pair of distinct elements $y,y'\in Y$.  Since $(HYH)^{-1}=HYH$, we can set $Y=\{v_1, ...,v_k,w_1, ...,w_k,z_1, ...,z_\ell\}$ such that $(Hv_iH)^{-1}=Hw_iH$ and $(Hz_jH)^{-1}=Hz_jH$ for all $1\leq i\leq k$ and $1\leq j\leq \ell$. Set $V_{i}=Hv_{i}H\cap X$, $W_{i}=Hw_{i}H\cap X$ and $Z_{j}=Hz_{j}H\cap X$ for all $1\leq i\leq k$ and $1\leq j\leq \ell$. Then
$X=(\dot{\cup}_{i=1}^{k}V_{i})\dot{\cup}(\dot{\cup}_{i=1}^{k}W_{i})\dot{\cup} (\dot{\cup}_{j=1}^{\ell}Z_{j})$. Since $HXH=XH$, we have
$HXH=(\dot{\cup}_{i=1}^{k}V_{i}H)\dot{\cup}(\dot{\cup}_{i=1}^{k}W_{i}H)\dot{\cup} (\dot{\cup}_{j=1}^{\ell}Z_{j}H)$.
Then, since $V_{1}\subseteq Hv_{1}H$ and $(X\setminus V_1)\cap Hv_{1}H=\emptyset$, we have $Hv_{1}H=V_{1}H$. Similarly, $Hv_{i}H=V_{i}H$, $Hw_{i}H=W_{i}H$ and $Hz_{j}H=Z_{j}H$ for all $1\leq i\leq k$ and $1\leq j\leq \ell$. Since $(V_{i}H)^{-1}=(Hv_{i}H)^{-1}=Hw_{i}H=W_{i}H$, we have $|V_{i}H|=|W_{i}H|$ and it follows that $|V_{i}|=|W_{i}|$. Therefore $m=|X|=2(\sum_{i=1}^{k}|V_{i}|)+\sum_{j=1}^{\ell}|Z_{j}|$. Since $m$ is an odd integer, we have $\ell\neq0$. Note that the inequality $\ell\neq0$ ensure the existence of $z_1$. Since $Hz_{1}^{-1}H=(Hz_1H)^{-1}=Hz_1H$, we have $z_{1}^{-1}H=hz_{1}H$ for some $h\in H$. It follows that $(z_{1}h)^{2}\in H$.
Since $Az_{1}A=Az_{1}^{-1}A$ and $z_1\in A\{g,g^{-1}\}A$, we have $Az_{1}A=AgA=Ag^{-1}A$. Therefore $z_1=bgc$ for some $b,c\in A$.
Set $x=gchb$. Then $x\in gA$ and
$x^{2}=gchbgchb=b^{-1}bgchbgchb=b^{-1}(z_{1}h)^{2}b\in b^{-1}Hb$.
\end{proof}
We leave it as an open problem whether the converse of Theorem \ref{necessity} holds. Now we give two corollaries of Theorem \ref{necessity}.
\begin{coro}
\label{notp}
Let $G$ be a group and $H\leq A \leq G$. If there exists an element $x\in G\setminus A$ such that $x^{2}\in A$, $xA$ contains no element whose square is contained in a conjugate of $H$ in $A$ and $|A:A\cap xAx^{-1}|$ is an odd integer, then $A$ is not a perfect code of $(G,H)$.
\end{coro}
\begin{proof}
Let $x$ be an element in $G\setminus A$ such that $x^{2}\in A$, $xA$ contains no element whose square is contained in a conjugate of $H$ in $A$ and $|A:A\cap xAx^{-1}|$ is an odd integer. Set $|A:A\cap xAx^{-1}|=m$. Since $axA=bxA$ if and only if $a^{-1}b\in xAx^{-1}$ for any $a,b\in A$, we have that $AxA=\dot{\cup}_{i=1}^{m}g_{i}A$ for some $g_{1},\ldots,g_{m}\in G$. Since $x^{2}\in A$, we have $A\{x,x^{-1}\}A=AxA=\dot{\cup}_{i=1}^{m}g_{i}A$. Since $m$ is an odd integer and $xA$ contains no element whose square is contained in a conjugate of $H$ in $A$, it follows from Theorem \ref{necessity} that $A$ is not a perfect code of $(G,H)$.
\end{proof}
\begin{coro}
\label{normal}
Let $G$ be a group, $H$ a subgroup of $G$ and $A$ a normal subgroup of $G$ such that $H \leq A \leq G$. If $A$ is a  perfect code of $(G,H)$, then for any $x \in G$ with $x^2 \in A$ there exists $b \in A$ such that $(xb)^2 \in H$.
\end{coro}
\begin{proof}
Let $A$ be a normal subgroup of $G$ and a perfect code of $(G,H)$. If $x\in A$, then $(xb)^2=1\in H$ where $b=x^{-1}\in A$. Now consider an arbitrary element $x\in G\setminus A$ with $x^2 \in A$. Since $A$ is normal in $G$, we have $A\{x,x^{-1}\}A=AxA=Ax^{-1}A=xA$. By Theorem \ref{necessity}, $xA$ contains an element $y$ satisfying $y^2\in b_{1}^{-1}Hb_1$ for some $b_1\in A$. Set $y=xa$ where $a\in A$. Since $A$ is normal in $G$, we have $x^{-1}b_1x\in A$. Set $b=x^{-1}b_1xab_{1}^{-1}$. Then $b\in A$.
Since $(xb)^2=(b_1xab_{1}^{-1})^2=(b_1yb_{1}^{-1})^2=b_1y^2b_{1}^{-1}$ and $y^2\in b_{1}^{-1}Hb_1$, we have $(xb)^2\in H$.
\end{proof}
The following result is a generalization of \cite[Theorem 3.7 (i)]{ZZ2021}.
\begin{theorem}
\label{quotient}
Let $G$ be a group and $H\leq A \leq G$. Let $N$ be a normal subgroup of $G$ which is contained in $A$. If $A$ is a perfect code of $(G,H)$, then $A/N$ is a perfect code of $(G/N,H/N)$.
\end{theorem}

\begin{proof}
Suppose that $A$ is a perfect code of $(G,H)$.
By Theorem \ref{lt}, there exists a left transversal $X$ of $A$ in $G$ such that $XH=HX^{-1}$. Since $N$ is a normal subgroup of $G$ and $N\leq A$, $X/N$ is a left transversal of $A/N$ in $G/N$.
Since $XH=HX^{-1}$, we have $(X/N)(H/N)=XH/N=HX^{-1}/N=(H/N)(X^{-1}/N)$.
By Theorem \ref{lt}, $A/N$ is a perfect code of $(G/N,H/N)$.
\end{proof}

\section{Subgroup total perfect codes of $(G,H)$}
\label{sec:tpc}

In this section, we deduce a few results on total perfect codes which are parallel to some results about perfect codes we obtained in Section \ref{sec:pc}.
\begin{theorem}
\label{totallt}
Let $G$ be a group and $H\leq A\leq G$. Then $A$ is a total perfect code of $(G,H)$ if and only if there exists a left transversal $X$ of $A$ in $G$ such that $XH=HX^{-1}$ and $X$ contains an element in $A\setminus H$.
\end{theorem}
\begin{proof}
$\Rightarrow$)  Suppose that $A$ is a total perfect code of $(G,H)$. By Lemma \ref{totallem}, $A$ is a perfect code of $(G,H)$ and there exists an element $x\in N_{A}(H)\setminus H$ such that $x^{2}\in H$. Then, by Theorem \ref{lt}, there exists a left transversal $Y$ of $A$ in $G$ such that $YH=HY^{-1}$. Since $x\in N_{A}(H)\setminus H$ and $x^{2}\in H$, we have $xH=H x^{-1}$. Since $Y$ a left transversal of $A$, $Y\cap A$ contains a unique element, say $y$. If $y\in A\setminus H$, then we set $X=Y$. If $y\in H$, then we set $X=(Y\setminus \{y\})\cup\{x\}$. In both cases, we have $XH=HX^{-1}$ and $X$ contains an element in $A\setminus H$.

$\Leftarrow$) Suppose that there exists a left transversal $X$ of $A$ in $G$ such that $XH=HX^{-1}$ and $X$ contains an element in $A\setminus H$. By Theorem \ref{lt}, $A$ is a perfect code of $(G,H)$. Since $X$ is a left transversal of $A$, $X\cap A$ contains a unique element. Set$X\cap A=\{x\}$ and $T=X\setminus\{x\}$. Since $H\leq A$, we have $xH$ and $Hx^{-1}$ are both contained in $A$. Therefore $TH\cap xH=TH\cap Hx^{-1}=\emptyset$. Then, since $XH=HX^{-1}$, we get $xH=Hx^{-1}$. Thus $x^{2}\in H$ and $x\in N_{A}(H)$. By Lemma \ref{totallem}, $A$ is a total perfect code of $(G,H)$.
\end{proof}
Similar to Theorem \ref{lt}, the word `left' in Theorem \ref{totallt} can be replaced by `right'. The proof of the following theorem is the same as that of Theorem \ref{rt} and therefore omitted.
\begin{theorem}
\label{totalrt}
Let $G$ be a group and $H \leq A \leq G$. Then $A$ is a total perfect code of $(G,H)$ if and only if there exists a right transversal $Y$ of $A$ in $G$ such that $Y^{-1}H=HY$ and $Y$ contains an element in $A\setminus H$.
\end{theorem}
The following theorem is a parallel result to Corollary \ref{sub}.
\begin{theorem}
\label{totalsub}
Let $G$ be a group and $H \leq A \leq L\leq G$. If $A$ is a total perfect code of $(G,H)$, then $A$ is a total perfect code of $(L,H)$.
\end{theorem}
\begin{proof}
Let $A$ be a total perfect code of $(G,H)$. By Lemma \ref{totallem}, $A$ is a perfect code of $(G,H)$ and there exists an element $x\in N_{A}(H)\setminus H$ such that $x^{2}\in H$. Since $H \leq A \leq L\leq G$,
it follows from Corollary \ref{sub} that $A$ is a perfect code of $(L,H)$. Then, since there exists an element $x\in N_{A}(H)\setminus H$ such that $x^{2}\in H$, Lemma \ref{totallem} ensures that $A$ is a total perfect code of $(L,H)$.
\end{proof}
The following theorem is the counterpart of Theorem \ref{KL} for total perfect codes. We omit its proof as it can be proceed by using a similar approach to the proof of Theorem \ref{totalsub}.
\begin{theorem}
\label{totalKL}
Let $G$ be a group admitting a normal subgroup $K$ and a subgroup $L$ such that $G=KL$ and $K\cap L=\{1\}$. Let $H\leq A\leq L$. Then $A$ is a total perfect code of $(G,H)$ if and only if $A$ is a total perfect code of $(L,H)$.
\end{theorem}
\section{Examples and problems}
\label{sec:ep}

In this section, we construct some examples and propose a few open problems.
\begin{problem}
Whether the converse of Theorem \ref{necessity} holds? If it does not hold, then what conditions should we add to make it true?
\end{problem}
For a normal subgroup $N$ of a given group $G$, it is straightforward to check that every coset graph $\Cos(G,N,U)$ is a Cayley graph on the quotient group $G/N$. Therefore for every subgroup $A$ of $G$ containing $N$, $A$ is a perfect code of $(G,N)$ if any only if $A/N$ is perfect code of $G/N$. By using this fact and Theorem \ref{KL}, we construct an infinite family of perfect codes as follows.
\begin{exam}
Let $G$ be the one dimensional affine group over a finite field $F$, and $L$ be the subgroup of $G$ consisting of elements in $G$ fixing the additive identity of $F$. Let $H\leq A\leq L$. Suppose that either the index $|L:A|$ of $A$ in $L$ is odd or the index  $|A:H|$ of $H$ in $A$ is odd. Then $A$ is perfect code of $(G,H)$.
\end{exam}
\begin{proof}
By \cite[Example 3.4.1]{DM1996}, $G$ is a Frobenius group, $L$ is a Frobenius complement of $G$ and $L$ is isomorphic to the cyclic group of order $|F|-1$. In particular, $H$ is normal in $L$ and the quotient group $L/H$ is cyclic. Since either $|L:A|$ or $|A:H|$ is odd, either $A/H$ is of odd order or $A/H$ is of odd index in $L/H$. By \cite[Corollary 2.8]{HXZ18}, $A/H$ is a perfect code in $L/H$. Therefore $A$ is perfect code of $(L,H)$. Let $K$ be the Frobenius  kernel of $G$. Then $G=KL$, $K\cap L=\{1\}$ and $K$ is normal in $G$. By Theorem \ref{KL}, $A$ is perfect code of $(G,H)$.
\end{proof}
Let $G$ be a group and $H\leq A\leq G$. If there exists a left transversal $X$ of $A$ in $G$ such that $XH=HX^{-1}$, then we call $(A,H,X)$ a \emph{perfect triple} of $G$.  By Theorem \ref{lt}, $A$ is a perfect code of $(G,H)$ if and only if there exists a subset $X$ of $G$ such that $(A,H,X)$ is a perfect triple of $G$. We propose the following problems for further research.
\begin{problem}
Given a group $G$ and its subgroup $H$, construct or classify the perfect triples $(A,H,X)$ of $G$.
\end{problem}
\begin{problem}
Given a group $G$ and its subgroup $A$, classify the perfect triples $(A,H,X)$ of $G$.
\end{problem}
Let $S_n$ be the symmetric group on $\{1,2,\ldots,n\}$. For two positive integers $m$ and $n$ with $m<n$, we treat $S_{m}$ as the subgroup of $S_n$ consisting of elements in $S_{n}$ fixing every number in $\{m+1,\ldots,n\}$.
\begin{exam}
\label{345}
Set $X=\{1,(2,5,3),(1,3,5),(1,5,2,3),(4,5)(1,3,2)\}$.
Then one can check that $(S_{4},S_{3},X)$ is a perfect triple of $S_{5}$. Therefore $S_{4}$ is perfect code of $(S_{5},S_{3})$.
\end{exam}
\begin{pro}
\label{symm}
Let $n$ be a positive integer at least $5$. Then $S_{n-1}$ is a perfect code of $(S_{n},S_{3})$.
\end{pro}
\begin{proof}
We proceed the proof by induction on $n$. By Example \ref{345}, the proposition is true for $n=5$. Suppose that the proposition is true for $n=k$ ($k\geq5$). It suffices to prove that the proposition is true for $n=k+1$. Let $G$ be the subgroup of $S_{k+1}$ consisting of elements in $S_{k+1}$ fixing $k$. Since $k\geq5$, we have $S_{3}< S_{k-1}< G$. Let $z$ be the involution $(k,k+1)$ in $S_{k+1}$. Then $z^{-1}S_{k}z=G$, $z^{-1}S_{3}z=S_{3}$ and $z^{-1}S_{k-1}z=S_{k-1}$. By induction hypothesis, $S_{k-1}$ is a perfect code of $(S_{k},S_{3})$. Therefore, by Corollary \ref{conj}, $S_{k-1}$ is a perfect code of $(G,S_{3})$. By Theorem \ref{lt}, there exists a left transversal $X$ of $S_{k-1}$ in $G$ such that $XS_3=S_3X^{-1}$. Set $Y=X\cup \{z\}$. Since $XS_3=S_3X^{-1}$ and $z^{-1}S_{3}z=S_{3}$, we obtain $YS_3=S_3Y^{-1}$. We will further prove that $Y$ is a left transversal of $S_k$ in $S_{k+1}$.

Since $X$ is a left transversal of $S_{k-1}$ in $G$, $x^{-1}y\notin S_{k-1}$ for each pair of distinct elements $x,y\in X$. Therefore $x^{-1}y$ does not fix $k+1$ and it follows that $x^{-1}y\notin S_{k}$. Thus $xS_{k}\neq yS_{k}$. Since $z=(k,k+1)$ and $x^{-1}$ fixes $k$ for every $x\in X$, $x^{-1}z$ takes $k+1$ to $k$. Therefore $x^{-1}z$ does not fix $k+1$ and it follows that $x^{-1}z\notin S_{k}$. Thus $xS_{k}\neq zS_{k}$.

From the above discussion, we have that $xS_{k}\neq yS_{k}$ for each pair of distinct elements $x,y\in Y$. Therefore $|YS_k|=|Y||S_{k}|=(k+1)k!=|S_{k+1}|$. Thus $YS_k=S_{k+1}$ and it follows that $Y$ is a left transversal of $S_k$ in $S_{k+1}$.

Now we have proved that $Y$ is a left transversal of $S_k$ in $S_{k+1}$ and $YS_3=S_3Y^{-1}$. By Theorem \ref{lt}, $S_{k}$ is a perfect code of $(S_{k+1},S_{3})$. In other word, the proposition is true for $k+1$. Therefore the proposition is true for all positive integer $n\geq 5$.
\end{proof}
By the proof of Proposition \ref{symm}, if $(S_{k-1}, S_3, X)$ is a perfect triple of $S_{k}$, then $(S_{k}, S_3, z^{-1}Xz\cup \{z\})$ is a perfect triple of $S_{k+1}$ where $z=(k,k+1)$. This provides us with a way to construct perfect triples $(S_{n-1}, S_3, X)$ of $S_{n}$ for every $n\geq 5$. Based on Example \ref{345}, the following two examples are constructed through this approach.
\begin{exam}
\label{356}
Set
\begin{equation*}
X=\{1,(2,6,3),(1,3,6),(1,6,2,3),(4,6)(1,3,2),(5,6)\}.
\end{equation*}
Then $(S_{5},S_{3},X)$ is a perfect triple of $S_{6}$.
\end{exam}
\begin{exam}
\label{367}
Set
\begin{equation*}
X=\{1,(2,7,3),(1,3,7),(1,7,2,3),(4,7)(1,3,2),(5,7),(6,7)\}.
\end{equation*}
Then $(S_{6},S_{3},X)$ is a perfect triple of $S_{7}$.
\end{exam}
More generally, we have the following proposition.
\begin{pro}
\label{3n}
Let $n\geq 3$ be a positive integer. Set
\begin{equation*}
X=\{1,(2,n,3),(1,3,n),(1,n,2,3),(4,n)(1,3,2),(5,n),\ldots,(n-1,n)\}.
\end{equation*}
Then $(S_{n-1},S_{3},X)$ is a perfect triple of $S_{n}$.
\end{pro}

\noindent {\textbf{Acknowledgements}}~~The first author was supported by the National Natural Science Foundation of China (No.~12071312). The second author was supported by the National Natural Science Foundation of China (No.~11671276), the Basic Research and Frontier Exploration Project of Chongqing (cstc2018jcyjAX0010) and the Foundation of Chongqing Normal University (21XLB006).
\end{document}